\documentclass[11pt, reqno]{amsart}

\usepackage{geometry}
\usepackage{color}
\usepackage{amssymb}
\usepackage{amsmath}
\usepackage{arydshln}
\usepackage{hyperref}
\usepackage{bbm}
\usepackage{blindtext}
\blindmathtrue
\usepackage{mathrsfs}
\usepackage{enumerate}

\usepackage{graphicx}
\usepackage{subfig}
\numberwithin{equation}{section}

\newcommand{\R}{\mathbb{R}}

\newcommand{\N}{\mathbb{N}}

\renewcommand{\P}{\mathbb{P}}
\newcommand{\E}{\mathbb{E}}

\newcommand{\cB}{\mathcal{B}}

\newcommand{\cU}{\mathcal{U}}
\newcommand{\cX}{\mathcal{X}}

\newtheorem{Theorem}{Theorem}[section]
\newtheorem{Proposition}[Theorem]{Proposition}
\newtheorem{Corollary}[Theorem]{Corollary}

\newtheorem{Remark}[Theorem]{Remark}
\newtheorem{Definition}[Theorem]{Definition}

\begin{document}

\title{Wasserstein distance in terms of the comonotonicity Copula}
\author{Mariem Abdellatif}
\address[Mariem Abdellatif]{School of Mathematics and Natural Sciences\\ University of Wuppertal, Germany}
\email[Mariem Abdellatif]{abdellatif@uni-wuppertal.de}

\author{Peter Kuchling}
\address[Peter Kuchling]{School of Mathematics and Natural Sciences\\ University of Wuppertal, Germany}
\email[Peter Kuchling]{kuchling@uni-wuppertal.de}

\author{Barbara R\"udiger}
\address[Barbara R\"udiger]{School of Mathematics and Natural Sciences\\ University of Wuppertal, Germany}
\email[Barbara R\"udiger]{ruediger@uni-wuppertal.de}

\author{Irene Ventura}
\address[Irene Ventura]{University of Trento, Italy}
\email[Irene Ventura]{irene.ventura@studenti.unitn.it}

\date{\today}

\begin{abstract}
In this article, we represent the Wasserstein metric of order $p$, where $p\in [1,\infty)$, in terms of the comonotonicity copula, for the case of probability measures on $\R^d$, by revisiting existing results.
 %We investigate in this work the link between the Wasserstein metric of order $p$, where $p\in [1,\infty)$ and Copula. 
 In \cite{vallander_ru}, Vallender established the link between the $1$-Wasserstein metric and the corresponding distribution functions for $d=1$. In \cite{DALL’AGLIO} Giorgio dall'Aglio showed that the $p$-Wasserstein metric in $d=1$ could be written in terms of the comonotonicity copula $M$ without being aware of the concept of copulas or Wasserstein metrics. In this article,
 %we show how the $p$-Wasserstein metric can be written in terms of the copula $M$ on $\R^d$
 for the proofs we explicitly combine tools from copula theory and Wasserstein metrics. The extension to general $d\in\N$ has some restriction, as discussed e.g. in \cite{Alfonsi} and \cite{BDS}. Some of the results of \cite{Alfonsi}, \cite{BDS} and \cite{RR} are revisited here in a more explicit form in terms of the comonotonicity copula.
\end{abstract}


\maketitle
\allowdisplaybreaks

\section{General introduction}
Wasserstein distances or Kantorovich–Rubinstein distances have been introduced first by Leonid Kantorovich in $1939$ \cite{MR129016} and they are used in many areas of pure and applied mathematics. The concept of Wasserstein distance is motivated by the concept of optimal transportation and it is based on finding an appropriate coupling between two marginal probability measures. Indeed, the optimal transport cost between two probability measures $\mu$ and $\nu$ on a set $\mathcal{X}$ is defined by 
\begin{align}\label{optimal transport cost}
    C(\mu,\nu)=\inf_{\pi \in \Pi(\mu,\nu)} \int_{\mathcal{X}\times\mathcal{X}} c(x,y) \pi(dx,dy),
\end{align}
where $c(x,y)$ is the cost for transporting one unit of mass from $x$ to $y$, $\Pi(\mu,\nu)$ is the set of all couplings between $\mu$ and $\nu$, i.e. the set of all probability measures on $\mathcal{X}\times\mathcal{X}$ with margins $\mu$ and $\nu$. In fact, when the cost is defined in terms of a distance on a Polish space $(\cX,d)$, then one can prove that \eqref{optimal transport cost} actually defines a distance, see e.g. \cite{villani_topics}. %If we take two random variables $X$ and $Y$ on a probability space $(\Omega, \mathcal{F}, P)$ with distribution functions $F$ and $G$ and taking values in a 

More precisely, let $\mu,\nu$ be two probability measures on a Polish space $(\mathcal{X},d)$, then the Wasserstein distance $ W_{p}$  of order $p\in [1,\infty)$ is defined by the following formula:
\begin{align}\label{Wm}
    W_{p}^{p}(\mu,\nu)=\inf_{\pi \in \Pi(\mu,\nu)} \int_{\mathcal{X}} d(x,y)^{p} \pi(dx, dy).
\end{align}
The Wasserstein space of order $p$ is the space of probability measures which have a finite moment of order $p$ and is defined as 
\begin{align*}
     P_{p}(\mathcal{X}):=\{\mu \in P(\mathcal{X}); \int_{\mathcal{X}} d(x_{0},x)^{p} \mu(dx)<\infty\},
\end{align*}
where $x_{0} \in \mathcal{X}$ is arbitrary and $P(\mathcal{X})$ is the set of all probability measures on $\mathcal{X}$. It turns out that for any $p\in[1,\infty)$, $W_p$ defines a metric on $P_p(\cX)$ \cite[Theorem 7.3]{villani_topics}.

%Unlike the relative entropy or the Kullback-Leibler divergence between two probability measures, 
The Wasserstein distance provides a meaningful and smooth representation of the distance between distributions. Furthermore, it is related to the notion of weak convergence of measures (see e.g. \cite[Section 6]{Villani}) and it has various applications in stochastic analysis, especially in ergodicity theory, see for example %\cite{arxiv.2301.05120} for a theoretical background and  
\cite{Friesen,MR4241464,MR4153590,arxiv.2301.05640,arxiv.2301.05120} for applications.
Remark that the following theorem guarantees that the infimum  in \eqref{Wm} is reached by some optimal coupling.
\begin{Theorem}[{Existence of optimal coupling, \cite[Thm. 4.1]{Villani}}]\label{Existence of optimal coupling}
Let $(\cX,d)$ be a Polish space and $\mu,\nu \in P_{p}(\cX)$. Then there exists a coupling $\pi \in \Pi(\mu,\nu)$ such that
\begin{align*}
   W_{p}^{p}(\mu,\nu)= \int_{\cX\times\cX} d(x,y)^{p} \pi(dx, dy).
\end{align*}
\end{Theorem}
We should hence expect that at least in some cases, this infimum can be identified by writing the couplings in terms of ``copulas''.
Copulas are functions that join or “couple” distribution functions to obtain multivariate distribution functions with prescribed marginals given by the distribution functions. They are used to describe and analyse dependence between random variables. The first explicit introduction was given by Abe Sklar in $1959$ \cite{Nelsen}. Giorgio dall'Aglio \cite{DALL’AGLIO} was the first to link the Wasserstein distance to copula, without being aware of the concept of Wasserstein metric or copula explicitly. This connection seems natural, as the Wasserstein distance itself is defined by minimizing over all couplings between two marginal distributions. While copulas are a commonly employed tool in actuarial sciences to analyse and simulate dependence structures of risks \cite{denuit2005actuarial,MR1968943,MR2244349,MR3445371}, its uses in pure mathematics have often been overlooked.\\
In \cite{vallander_ru}, Vallender established the link between the $1$-Wasserstein metric and the distribution functions.

One aim of this article is to reformulate the results from \cite{DALL’AGLIO} and \cite{vallander_ru} in terms of copula explicitly. This is especially reflected in the proof of our first main theorem (Theorem  \ref{main theorem}), where we employ copula theory to first express the Wasserstein distance on $\R$ in terms of copula, and then rewrite the expression to regain a classical result in terms of generalized inverse distribution functions. Remark that the copula employed is the comotonicity copula, which is a maximizer  for all copulas according to the Theorem of Fr\'echet-Hoeffding, recalled in Section \ref{sec:prelim}.

Our second aim is to rewrite the result to the case of $\R^d$ obtained in \cite{Alfonsi}, \cite{BDS} and \cite{RR} explicitly for the case where the underlying common copula is the comotonicity copula. In \cite{RR}, a general characterization of the optimal coupling of the Wasserstein distance was given in terms of convex analysis, cf. \cite[Theorem 3.2.9 and Section 3.3]{RR}, while also providing the special case $d=1$ as an example (\cite[Theorem 3.1.2 and Example 3.2.14]{RR}). As discussed in \cite{Alfonsi} and \cite{BDS} there are restrictions for which the infimum in the Wasserstein metrics \eqref{Wm} can be identified in terms of copula, according to Theorem \ref{Existence of optimal coupling}. It turns out that only in the restricted case where both measures share the same dependence structure, the Wasserstein distance can be expressed in terms of the shared copula as well as the margins. In other words, for the result in \cite{Alfonsi} it is necessary that the probability measures $\mu$ and $\nu$ share the same copula $C$. For the explicit result, in this article we assume that the two probability measures share the same copula $M$ which will be defined below. To obtain an explicit representation, the coupling needed is again the copula which serves as the minimizer in the case $d=1$, which is exactly the comonotonicity copula $M$. To emphasize the importance of the comonotonicity copula, we give an explicit proof of our main theorem for the case that the underlying dependence structure is comonotone. Let us remark that various extensions to the concept of comonotonicity and their relation to optimality have been discussed in \cite{MR2557634}.

The article is structured as follows. In Section \ref{sec:main}, we formulate the main results of this article and compare our work to some previous important results done in \cite{BDS} and \cite{Alfonsi}. In Section \ref{sec:prelim}, we recall the notion of coupling as well as some classical results from probability theory on $\R^d$. To keep the article self-contained, we also give a basic introduction into the fundamental results of copula theory. Section \ref{sec:main_proofs} is devoted to the proofs of the main results. For completeness, we recall some proofs of more basic results in Appendix \ref{app:a}.

%Structure of the work in our main result of this work we study the link between the p-Wasserstein metric and the comonotonicity copula M . In the first chapter we set up our two main theorems: The first and the second theorems will be devoted to calculating the p-Wasserstein metric in the case X = R with the usual Euclidean distance and for the case X = Rd, d > 1 with the p-norm, respectively. The second chapter will be consecrated to remind some necessary preliminary results that we will use throughout the article. In the last chapter we give the proof of the two main results.

\section{Main results and discussion}\label{sec:main}

Our first result is the representation of the Wasserstein distance on $\R$ in terms of the comonotonicity copula. As it turns out, the optimal coupling is always given by the two-dimensional comonotonicity copula, called here also $M$-copula. Here and below, we denote by $M^d\colon [0,1]^d\to[0,1]$,
\begin{equation}\label{Comonotinicity copula}
 M^d(u_1,\dotsc,u_d)=\min\{u_1,\dotsc,u_d\}
\end{equation}
where we omit the dimension index when no confusion may arise. We will explore the notion of copula in more detail in Section \ref{sec:prelim}.

\begin{Theorem}[{Wasserstein distance in terms of copula in $\R$, \cite{vallander_ru},
\cite[Teorema I, Teorema IX]{DALL’AGLIO}, \cite[Proposition 2.1]{MR1240428}}]\label{main theorem}
Let $\mu, \nu$ be two probability measures in $P_{p}(\R)$. %associated respectively to  two random variables $X$ and $Y$, 
Let $F$ and $G$ be the associated distribution functions. Then for all $p\in [1,\infty)$,
\begin{align*}
    W_{p}^{p}(\mu,\nu)&=\int_{-\infty}^{\infty} \int_{-\infty}^{\infty} |x-y|^{p} dM(F(x),G(y))
    \\
    &= \int_{0}^{1} |F^{-1}(u)-G^{-1}(u)|^{p} du,
\end{align*}
where %$M$ is the Copula defined by $M(F(x),G(y)):=\min\{F(x),G(y)\}$, 
$F^{-1}$ and $G^{-1}$ are the generalized inverses (or the quantile function of $F$ and $G$) associated to $F$ and $G$ respectively and defined as 
\begin{align*}
    F^{-1}(u)=\inf\{x \in \R \colon F(x)\geq u\}, u \in [0,1],
\end{align*}
and 
\begin{align*}
    G^{-1}(u)=\inf\{x \in \R \colon G(x)\geq u\}, u \in [0,1],
\end{align*}
with the convention that $\inf \emptyset =\infty$.
\end{Theorem}
\begin{Remark}
 \begin{enumerate}
     \item The representation via the generalized inverses on $\R$ is a well-known fact in optimal transport, see e.g. \cite[Theorem 2.18 and Remark 2.19]{villani_topics}.
     \item In \cite{vallander_ru}, the author considered only the case $p=1$, while \cite{MR1240428} discusses the case $p=2$. In \cite{DALL’AGLIO}, all cases $p=1$ and $p>1$ were considered. Nevertheless, the author also did not draw the connection to copula theory. We will revisit the methods of their proofs in Section \ref{sec:main_proofs} while combining it with the language of copulas.
     \item As the author also noticed in \cite{vallander_ru}, for $p=1$, the Wasserstein distance on $\R$ may even be expressed in terms of the distribution functions $F$ and $G$:
     \begin{displaymath}
      W_1(\mu,\nu)=\int_\R|F(x)-G(x)|dx.
     \end{displaymath}
 \end{enumerate}
\end{Remark}

Only for some restricted case, the result can be generalized to $\R^d$. This is discussed in \cite{Alfonsi} (see also \cite{BDS} and \cite{RR}). In fact, to obtain a representation of the optimal coupling using copulas and the marginals, it is required that both probability measures share the same dependence structure, i.e., the same copula. This was proven in the finite-dimensional case in \cite[Proposition 1.1]{Alfonsi} and \cite{RR} and was generalized to the infinite-dimensional case in \cite[Theorem 5]{BDS}. To obtain an explicit representation, we focus on the case that the underlying shared copula is the $d$-dimensional comonotonicity copula. Denote by $\|\cdot\|_p$ the $p$-norm on $\R^d$. Since all norms on $\R^d$ are equivalent, it suffices to consider the one with the same order as the Wasserstein distance in order to analyse convergence properties. Note though, that an explicit representation as given below is not possible if the order of the Wasserstein distance and the norm on $\R^d$ do not coincide, as was remarked in \cite[Proposition 1.1]{Alfonsi}.

Let $\mu,\nu\in P(\R^d)$ be two probability measures and $(F_1,\dotsc,F_d)$ and $(G_1,\dotsc,G_d)$ their marginal distribution functions. We say that $\mu$ and $\nu$ share the same copula if there exists a $d$-dimensional copula $C\colon[0,1]^d\to[0,1]$ such that the joint distribution functions $H_\mu$ and $H_\nu$ of $\mu$ and $\nu$ can be written as follows:
\begin{equation}\label{eq:shared_copula}
\begin{split}
 H_\mu(x_1,\dotsc,x_d)&=C(F_1(x_1),\dotsc,F_d(x_d))
 \\
 H_\nu(y_1,\dotsc,y_d)&=C(G_1(y_1),\dotsc,G_d(y_d)).
\end{split}
\end{equation}
For more details on the connection between distribution functions and copulas, see Section \ref{sec:prelim}. For such measures with \eqref{eq:shared_copula}, we may represent the Wasserstein distance in terms of copula. For another proof in the finite-dimensional case, see \cite{Alfonsi}. The following representation theorem was proven for the infinite-dimensional case in \cite{BDS}. For better comparison with our results, we formulate a finite-dimensional version here. It can also be found in \cite[Theorem 2.9]{MR1240428}.
\begin{Theorem}[{\cite[Theorem 5]{BDS}, \cite[Theorem 2.9]{MR1240428}}, formulation for finite-dim. case]\label{thm:bds}
Let $X,Y$ be random variables on $\R^d$. Then the following are equivalent:
\begin{enumerate}
    \item $X$ and $Y$ share the same copula $C$.
    %\item The distribution  of $(F_X^{-1}(U),F_Y^{-1}(U))$ is an optimal coupling of $X$ and $Y$, where $U$ is a $\R^d$-dimensional random variable with $\cU(0,1)$-margins whose distribution function is given by $C$.
    \item The Wasserstein distance between $X$ and $Y$ is given by
    \begin{displaymath}
     W_p^p(X,Y)=\sum_{i=1}^dW_p^p(X_i,Y_i)
    \end{displaymath}
\end{enumerate}
In particular, if one of the above holds, we have
\begin{displaymath}
 \sum_{i=1}^dW_p^p(X_i,Y_i)=\int_{[0,1]^d}\sum_{i=1}^d|F_i^{-1}(u_i)-G_i^{-1}(u_i)|^pdC(u_1,\dotsc,u_d).
\end{displaymath}
\end{Theorem}

\begin{Theorem}[Wasserstein distance in terms of copula in $(\R^{d},\|\cdot\|_p)$, $d>1$ \cite{Alfonsi}, \cite{BDS}]\label{main theorem 2}
Let $\mu, \nu$ be two probability measures in $P_{p}(\R^d)$ which share the same copula. %associated respectively to two random vectors $X=(X_{1},X_{2},...,X_{d})$ and $Y=(Y_{1},Y_{2},...,Y_{d})$, let $F$ and $G$ be the associated distribution functions. 
Denote by $F_i$ and $G_i$ ($i=1,\dotsc,d$) the distribution functions of the one-dimensional margins of $\mu$ and $\nu$, respectively. Then for all $p\in [1,\infty)$, 
\begin{align*}
    W_{p}^{p}(\mu,\nu) &= \int_{[0,1]} \left \|F^{-1}(u)-G^{-1}(u)\right \|_p^{p} du, 
\end{align*}
where $F^{-1}, G^{-1}:[0,1] \rightarrow \R^{d}$ are the generalized inverses associated to $F$ and $G$ respectively, where for all $u \in [0,1]$
\begin{align*}
    F^{-1}(u)&:= \left( F_{1}^{-1}(u),\dotsc,F_{d}^{-1}(u) \right) \\ G^{-1}(u)&:=\left( G_{1}^{-1}(u),\dotsc,G_{d}^{-1}(u) \right)
\end{align*}
and $F_{i}^{-1}, G_{i}^{-1}, i=1,\dotsc,d$, are one-dimensional generalized inverses defined by,
\begin{align*}
    F^{-1}_{i}(u)=\inf\{x \in \R \colon F_{i}(x)\geq u\}, u \in [0,1],
\end{align*}
and 
\begin{align*}
    G^{-1}_{i}(u)=\inf\{x \in \R \colon G_{i}(x)\geq u\}, u \in [0,1].
\end{align*}
\end{Theorem}

Our following result aims to rewrite the above representation of Theorem \ref{main theorem 2} in terms of the comonotonicity copula.

\begin{Proposition}[Wasserstein distance in terms of the comonotonicity copula $M$ in $(\R^{d},\|\cdot\|_p)$, $d>1$]\label{Representation for measures with the copula M}
Let $\mu,\nu$ be two probability measures in $P_p(\R^d)$ which both have the $M$-copula, i.e., \eqref{eq:shared_copula} can be written with
\begin{displaymath}
 C(z_1,\dotsc,z_d)=M(z_1,\dotsc,z_d)=M^d(z_1,\dotsc,z_d)=\min\{z_1,\dotsc,z_d\}.
\end{displaymath}
Denote by $F_i$ and $G_i$ ($i=1,\dotsc,d$) the distribution functions of the one-dimensional margins of $\mu$ and $\nu$, respectively. Then for all $p\in [1,\infty)$, 
\begin{align*}
    W_{p}^{p}(\mu,\nu)&=\sum_{i=1}^{d} \int_{\R} \int_{\R} |x_{i}-y_{i}|^{p} dM(F_{i}(x_{i}),G_{i}(y_{i})), 
\end{align*}
where $M$ is the copula defined by \ref{Comonotinicity copula}.
\end{Proposition}

From the result we proved in Proposition \ref{Representation for measures with the copula M} we can get directly by combination with Theorem \ref{thm:bds} the statement obtained by A. Alfonsi and B. Jourdain in Theorem \ref{main theorem 2}. (See proof of Theorem \ref{thm:to_be_deleted} in Section \ref{sec:main_proofs}). The final statement is the following:

%By combining Theorem \ref{main theorem 2} and Proposition \ref{Representation for measures with the copula M} together, we get the following result:

\begin{Theorem}\label{thm:to_be_deleted}
Let $\mu,\nu$ be two probability measures in $P_p(\R^d)$ sharing the same copula. Denote by $F_i$ and $G_i$ ($i=1,\dotsc,d$) the distribution functions of the one-dimensional margins of $\mu$ and $\nu$, respectively. Then for all $p\in [1,\infty)$, 
\begin{align*}
    W_{p}^{p}(\mu,\nu)&=\sum_{i=1}^{d} \int_{\R} \int_{\R} |x_{i}-y_{i}|^{p} dM(F_{i}(x_{i}),G_{i}(y_{i}))
    \\ 
    &= \int_{[0,1]} \left \|F^{-1}(u)-G^{-1}(u)\right \|_p^{p} du, 
\end{align*}
where $M$ is the copula defined by \ref{Comonotinicity copula}, $F^{-1}, G^{-1}:[0,1] \rightarrow \R^{d}$ are the generalized inverses associated to $F$ and $G$ respectively, where for all $u \in [0,1]$
\begin{align*}
    F^{-1}(u)&:= \left( F_{1}^{-1}(u),\dotsc,F_{d}^{-1}(u) \right) \\ G^{-1}(u)&:=\left( G_{1}^{-1}(u),\dotsc,G_{d}^{-1}(u) \right)
\end{align*}
and $F_{i}^{-1}, G_{i}^{-1}, i=1,\dotsc,d$, are one-dimensional generalized inverses defined by
\begin{align*}
    F^{-1}_{i}(u)=\inf\{x \in \R \colon F_{i}(x)\geq u\}, u \in [0,1],
\end{align*}
and 
\begin{align*}
    G^{-1}_{i}(u)=\inf\{x \in \R \colon G_{i}(x)\geq u\}, u \in [0,1].
\end{align*}
\end{Theorem}

As all norms are equivalent on $\R^d$, we get the following result for $(\R^d,\|\cdot\|_q)$ for $q\neq p$:
\begin{Corollary}
 Let $\mu,\nu\in P_p(\R^d)$ and $q\geq 1$ with $q\neq p$ and assume the remaining assumptions and notation from Theorem \ref{main theorem 2}. Denote by $W_{p,q}$ the Wasserstein distance on $P_p(\R^d)$  on the space $(\R^d,\|\cdot\|_q)$. Then for any $\mu,\nu\in P_p(\R^d)$,
 \begin{displaymath}
  d^{-\frac{1}{p}}\int_0^1\|F^{-1}(u)-G^{-1}(u)\|_p^p du\leq W_{p,q}^p(\mu,\nu)\leq d^{\frac{1}{q}}\int_0^1\|F^{-1}(u)-G^{-1}(u)\|_p^p du.
 \end{displaymath}
\end{Corollary}

\begin{Remark}
What is remarkable about these statements is that we can reduce our analysis to one one-dimensional integral on $[0,1]$ instead of having to integrate on $\R^{2d}$, i.e. we use the same variable of integration for all coordinates. The representation via copula shows that the optimal coupling used for the Wasserstein distance corresponds to the case of comonotone random variables. We will discuss and use this concept of comonotonicity throughout Sections \ref{sec:prelim} and \ref{sec:main_proofs}, respectively.
\end{Remark}

\section{Preliminaries}\label{sec:prelim}
In order to prove the main theorem of this article, we set up some identities related to distribution functions and recall the fundamentals of copula theory. Note that both univariate and multivariate (or joint) distribution functions are understood in the probabilistic sense, i.e., we also assume right continuity.

\begin{Theorem}\cite[Thm 6.5.2]{Giorgio2003}.\label{tails}
Let $X$ be a random variable on $(\Omega, \mathcal{F}, P)$ with associated distribution function $F_{X}$, let $r>0$. If $\E[\lvert X\rvert^{r}]<\infty$ then
    \begin{equation}\label{conditions on the tails}
        \lim_{x\rightarrow +\infty} x^r [1-F_{X}(x)]=0=\lim_{x\rightarrow +\infty} x^rF_{X}(-x).
    \end{equation}
\end{Theorem}

For the definition of the Wasserstein distances, we need the notion of coupling. Recall that for probability measures $\mu,\nu$ on $\cB(\R^d)$, a probability measure $\pi$ on $\cB(\R^d\times\R^d)$ is called \emph{coupling} of $\mu$ and $\nu$ if for any Borel set $A\in\cB(\R^d)$,
\begin{displaymath}
 \pi(A\times\R^d)=\mu(A)\text{ and }\pi(\R^d\times A)=\nu(A).
\end{displaymath}
In other words, the marginals of $\pi$ are given by $\mu$ and $\nu$. %One classical result from optimal transport which is fundamental for explicit representations of the Wasserstein distance is the existence of a minimizing coupling. The following theorem makes this precise. \textcolor{green}{(The following theorem now also appears in the introduction, shall we delete it here? As far as I can see, we don't need it for the proofs.)}
%\begin{Theorem}[{Existence of optimal coupling, \cite[Thm. 4.1]{Villani}}]
%Let $\mu,\nu \in P_{p}(\R^d)$. Then there exists a coupling $\pi \in \Pi(\mu,\nu)$ such that
%\begin{align*}
%   W_{p}^{p}(\mu,\nu)= \int_{\R^d\times\R^d} \|x-y\|^{p} \pi(dx, dy).
%\end{align*}
%\end{Theorem}

Next, we introduce the notion of copula, which describes the dependence structure of a random vector independently of the individual behavior of the entries.
\begin{Definition}
A function $C\colon[0,1]^d\to[0,1]$ is called a ($d$-dimensional) copula if the following properties are fulfilled:
\begin{enumerate}
    \item For all $u\in[0,1]^d$ s.t. $u_i=0$ for some $i$, we have $C(u)=0$ (groundedness).
    \item For all $(a_1,\dotsc,a_d),(b_1,\dotsc,b_d)\in[0,1]^d$ with $a_i\leq b_i$ for all $i$, we have
    \begin{displaymath}
     \sum_{i_1=1}^2\dotsi\sum_{i_d=1}^2(-1)^{i_1+\dotsb i_d}C(u_{1i_1},\dotsc,u_{d i_d})\geq 0,
    \end{displaymath}
    where $u_{j1}=a_j$ and $u_{j2}=b_j$ for all $j\in\{1,\dotsc,d\}$ ($d$-increasing).
    \item $C(1,\dotsc,1,u_i,1,\dotsc,1)=u_i$ for all $i\in\{1,\dotsc,d\}$ and $u_i\in[0,1]$ (uniform margins).
\end{enumerate}
\end{Definition}
By inspecting the definition of copulas, one notices that they can be viewed as multivariate distribution functions on $[0,1]^d$ which have uniform margins on $[0,1]$. This is in accordance of the one-dimensional quantile transformation, which states that for a random variable $X$ with distribution function $F$ and its generalized inverse $F^{-1}$, we have $X\stackrel{d}=F^{-1}(U)$ for some random variable $U$ which is uniformly distributed on $[0,1]$.

A detailed introduction and analysis of copulas can be found in \cite{Nelsen}. We review some of the essential results below.

The following theorem is fundamental in the theory of copulas. It shows the relationship between multivariate distribution functions and their univariate margins.
\begin{Theorem}[{Sklar's theorem, \cite[Thm. 2.10.9]{Nelsen}}]\label{Sklar's theorem}
Let $H$ be a $d$-dimensional distribution function with margins $F_{1},\dotsc, F_{d}$. Then there exists a d-copula $C$ such that for all $x=(x_{1},\dotsc,x_{d}) \in\R^{d}$
\begin{equation}\label{eq:sklar_identity}
    H(x_{1},\dotsc,x_{d})=C(F_{1}(x_{1}),\dotsc,F_{d}(x_{d})).
\end{equation}
On the other hand, let $C$ be a copula and $F_1,\dotsc,F_d$ one-dimensional distribution functions. Then the function $H$ defined by \eqref{eq:sklar_identity} is a $d$-dimensional joint distribution function with margins $F_1,\dotsc,F_d$.
\end{Theorem}

One elementary result regarding joint distribution functions and copulas are the so-called Fr\'echet-Hoeffding bounds. They are defined as
\begin{align*}
    M^{d}(u_1,\dotsc,u_d) &= \min(u_1,\dotsc,u_d) 
    \\
    W^{d}(u_1,\dotsc,u_d) &= \max(u_1+u_2+\dotsb+u_d-d+1,0)
\end{align*}
While $M^d$ is a $d$-copula for any $d\geq 2$, this is not true for $W^d$ as soon as $d>2$. However, these functions do not only represent essential dependence structures, they also serve as elementary bounds for any other copula. More precisely, we have the following result:

\begin{Theorem}[{Fréchet-Hoeffding bounds in $d$ dimensions, \cite[Thm. 2.10.12]{Nelsen}}]\label{Fréchet-Hoeffding bounds}
    For any $d$-copula $C$ and for all $(x,y) \in \R^{2}$, 
    \begin{align*}
    W^{d}(F_{1}(x_{1}),\dotsc,F_{d}(x_{d})) &\leq C(F_{1}(x_{1}),\dotsc,F_{d}(x_{d})) \leq  M^{d}(F_{1}(x_{1}),\dotsc,F_{d}(x_{d})),
    \end{align*}
\end{Theorem} 
We omit the superscript $d$ in $M^d$ and $W^d$ when no confusion may arise. Due to the specific dependence structure given by $M$, it is also known as the comonotonicity copula. The following theorem makes this notion precise. We call a random vector $(X,Y)$ comonotonic if there exists a random variable $Z$ and nondecreasing functions $f,g$ such that
\begin{displaymath}
 (X,Y)\stackrel{d}=(f(Z),g(Z)).
\end{displaymath}
For $a<b$, denote by $\cU(a,b)$ the uniform distribution on the interval $(a,b)$.
\begin{Theorem}[{Equivalent conditions comonotonicity, \cite[Thm. 3]{DDGKV2001}}]\label{Equivalent conditions comonotonicity}
Let $(\Omega, \mathcal{F}, \P)$ be a probability space, $X,Y$ be two $\R$-valued random variables on $(\Omega, \mathcal{F}, \P)$ with distribution functions $F$ and $G$, respectively and joint distribution functions $H$. A random vector $(X,Y)$ is comonotonic if and only if one of the following equivalent conditions holds:
    \begin{enumerate}
     \item For all $(x,y) \in \R^{2}$, we have
        \begin{align*}
            H(x,y)= \min \{F(x), G(y)\}=M(F(x),G(y)).
        \end{align*}
        \item For $U \sim \cU(0,1)$, we have
        \begin{align*}
            (X,Y)\overset{d}{=} (F^{-1}(U), G^{-1}(U)).
        \end{align*}
    \end{enumerate}
\end{Theorem}

The following auxiliary statement is used in the proofs of our main results. Its proof is given in the appendix for smooth reading.
%\textcolor{blue}{(we delete lemma 3.6)}
%\begin{Lemma}\label{Equivalence}
%Let $X$ be an $\R$-valued random variable with distribution function $F$. For all $x \in \R$ and $p\in [0,1]$, we have
%\begin{align}\label{Generalised inverse}
%    F^{-1}(p) \leq x \Leftrightarrow  p\leq F(x)
%\end{align}
%\end{Lemma}

\begin{Proposition}[\cite{Nelsen}, Example $5.1$ ]\label{comonotonicity result}
 Let $X,Y$ be two random variables s.t. $(X,Y)$ is a comonotonic random vector. Then for any function $g$ such that $g(X,Y)\in L^1(\P)$,
 \begin{equation*}
  \E[g(X,Y)]=\int_{0}^1 g(F^{-1}(u),G^{-1}(u)) du.
 \end{equation*}
 where $F,G$ denote the distribution functions of $X$ and $Y$, respectively.
\end{Proposition}

\begin{Remark}
There are two equivalent definitions regarding the notion of Wasserstein distance. Namely, for $\mu,\nu\in P_p(\R^d)$, we have
\begin{displaymath}
 W_p^p(\mu,\nu)=\inf_{\pi\in\Pi(\mu,\nu)}\int_{\R^d\times\R^d}\|x-y\|^p\pi(dx,dy)
\end{displaymath}
where the expression is minimized over all couplings of $\mu,\nu$. On the other hand, we may express $W_p$ in terms of random vectors:
\begin{displaymath}
 W_p^p(\mu,\nu)=\inf_{X,Y}\E\big[\|X-Y\|^p\big]
\end{displaymath}
where the infimum is taken over all random variables $X,Y$ with $X\sim\mu$ and $Y\sim\nu$. We will use both definitions in the proofs, whichever is most convenient in any given situation.
\end{Remark}

\section{Proofs of Theorems \ref{main theorem}, \ref{thm:to_be_deleted} and Proposition \ref{Representation for measures with the copula M}}\label{sec:main_proofs}

We are now ready to prove the main results of this work. We start with the proof of Theorem \ref{main theorem}, which is then used to prove the multi-dimensional extension in Theorem \ref{main theorem 2}.

The proof of Theorem \ref{main theorem} will be divided into two parts: first of all we will prove the following one-dimensional result in Proposition \ref{prop:p4.1} for $p=1$ using partially ideas from the  article by Vallender \cite{vallander_ru}, then we consider $p>1$ using partially ideas from the book by dall'Aglio \cite{DALL’AGLIO}. The final result in Proposition \ref{prop:p4.1} can be found in both articles. Here we combine however the theory of Wasserstein metrics and copula for the proof.
\begin{Proposition}\label{prop:p4.1}
    Let $\mu, \nu \in P_{1}(\R) $, $F$ and $G$ be the associated distribution functions. Then
    \begin{align*}
        W_{1}(\mu,\nu)&= \int_{-\infty}^{\infty} |F(x)-G(x)| dx
        \\
        &=\int_{0}^{1} |F^{-1}(u)-G^{-1}(u)| du.
    \end{align*}
\end{Proposition}
\begin{proof}
 By \cite[Theorem 14.1]{Billingsley}, we get that for $F$ and $G$ two distribution functions there exists on some probability space $(\Omega, \mathcal{F}, \P)$ two random variables $X$ and $Y$  such that $X\sim \mu$ and $Y \sim \nu$. Since $\mu,\nu\in P_1(\R)$, we have $W_{1}(\mu,\nu)<\infty$ and $\E|X-Y|<\infty$. Under these conditions, it was proven in \cite{vallander_ru} that 
 \begin{align*}
  \E|X-Y| &= \int_{-\infty}^{\infty} \left[ P(X\leq y)+P(Y\leq y)-2P(X\leq y, Y\leq y)  \right] dy
  \\
  &= \int_{-\infty}^{\infty} [F(y)+G(y)-2H(y,y)]dy
 \end{align*}
 where $H$ is the joint distribution function of the random variables $X$ and $Y$.
 
 Denoting by $C$ the copula corresponding to $X$ and $Y$, by Theorems \ref{Sklar's theorem} and \ref{Fréchet-Hoeffding bounds}, we get
    \begin{align*}
         \int_{-\infty}^{\infty} [F(y)+G(y)-2H(y,y)]dy &= \int_{-\infty}^{\infty} [F(y)+G(y)-2C(F(y),G(y))]dy
         \\
         &\geq  \int_{-\infty}^{\infty} [F(y)+G(y)-2M(F(y),G(y))]dy
    \end{align*}
    Taking the infimum over all possible $X\sim\mu$ and $Y\sim\nu$, this yields
    \begin{displaymath}
     W_1(\mu,\nu)\geq\int_{-\infty}^\infty [F(y)+G(y)-2M(F(y),G(y))]dy.
    \end{displaymath}
    It is left to show the inequality in the other direction. To this end, let $(\tilde{X},\tilde{Y})$ be a comonotone random vector with distribution $\mu$ and $\nu$, respectively. By Theorem \ref{Equivalent conditions comonotonicity}, its copula is given by $M$. Therefore,
    \begin{align*}
       W_1(\mu,\nu)=\inf_{\substack{X\sim\mu,\\Y\sim\nu}} \E |X-Y| \leq \E|\tilde{X}-\tilde{Y}| = \int_{-\infty}^{\infty} [ F(y)+G(y)-2M(F(y), G(y))] dy.
   \end{align*}
    
    %Combining the above estimate with Theorem \ref{Existence of optimal coupling}, we find the optimal coupling which minimize $\E|X-Y|$ which is actually the copula $M$, i.e.,
    Putting these inequalities together, we arrive at
   \begin{align*}
       W_1(\mu,\nu) = \E|\tilde{X}-\tilde{Y}| = \int_{-\infty}^{\infty} \left[ F(y)+G(y)-2M(F(y), G(y))\right] dy.
   \end{align*}
   %for a comonotonic pair $(\tilde{X},\tilde{Y})$.
   Hence, by the definition of $M$, we get
    \begin{align*}
      W_{1}(\mu,\nu)= \int_{-\infty}^{\infty} |F(y)-G(y)| dy.
    \end{align*}
    Furthermore, it follows from Proposition \ref{comonotonicity result} that
    \begin{align*}
        W_{1}(\mu,\nu)= E|\tilde{X}-\tilde{Y}|&= \int_{0}^{1} |F^{-1}(u)-G^{-1}(u)| du.
    \end{align*}
\end{proof}

Now we move to the case where $p>1$. Here, we will use previous results by dall'Aglio \cite{DALL’AGLIO}. A similar approach may also be used to prove the case $p=1$, but the proof of Proposition \ref{prop:p4.1} illustrates nicely the significance of the $M$-copula together with comonotonicity of the corresponding random variables.

%Actually we could also skip the proof by Vallender for $p=1$ and only use the proof the by Giorgio dall'Aglio for all $p\geq 1$ but we wanted to show explicitly the existence of the copula $M$\\

Fix $\mu,\nu$ with associated distribution functions $F$ and $G$. In spirit of \cite{DALL’AGLIO}, for $p>1$, define
\begin{align*}
   I(H):= \int_{-\infty}^{\infty} \int_{-\infty}^{\infty} \lvert y-x\rvert ^{p} dH(x,y) 
\end{align*}
 which is the expression to be minimized in the Wasserstein metric, i.e.,
\begin{displaymath}
 W_p^p(\mu,\nu)=\inf_{H}\int_{-\infty}^\infty\int_{-\infty}^\infty|y-x|^pdH(x,y)
\end{displaymath}
where the infimum runs over all bivariate distribution functions $H$ with margins $F$ and $G$.

By assuming \eqref{conditions on the tails} on the tails of the margins $F$ and $G$, dall'Aglio came to prove that the double integral above can be minimized. By Theorem \ref{tails}, it is sufficient for \eqref{conditions on the tails} to hold to have moment assumptions of the corresponding order. While this seems like a restriction at first glance, note that we are working on the Wasserstein space of order $p\in[1,\infty)$, which already implies the existence of moments of order $p$.

%Now from Theorem \ref{tails} since dall'Aglio was not aware by the theory of Wasserstein metric, then by just assuming that the two probability measures $\mu,\nu$, associated to $F$ and $G$ respectively, are in the Wasserstein sapce $ P_{p}(\R)$, then the conditions on the tails which are mentioned in the proof by dall'Aglio are fulfilled and we can set up the following:
\begin{Proposition}[{\cite[Equation (30)]{DALL’AGLIO}}]\label{First proposition by Giorgio dall'aglio}
Let $\mu,\nu \in P_{p}(\mathcal{\R})$ with associated distribution functions $F$ and $G$, respectively, and $H$ be a joint distribution function with margins $F$ and $G$. Then
\begin{equation}\label{Dall'Aglio equality}
\begin{split} 
    I(H)&= p (p -1)\int_{-\infty}^{\infty} \int_{y}^{\infty}[G(y)-H(x,y)](x-y)^{p -2} dxdy
    \\
    &\hspace{20pt}+ p (p -1)\int_{-\infty}^{\infty} \int_{x}^{\infty}[F(x)-H(x,y)](y-x)^{p -2} dydx.
\end{split}
\end{equation}
\end{Proposition}

%\begin{Remark}
 % In order to prove \ref{tails equation} we need additional condition on the tails but by only assuming that 
%\end{Remark}
The proposition now provides a useful representation of $I$ to find the minimizer. The statement below was proven by Girogio dall'Aglio, however ignoring the concept of Wasserstein metrics and Wasserstein  space $P_{p}(\R)$. Our proof refers however explicitly to these concepts.
\begin{Theorem}[{\cite[Teorema IX]{DALL’AGLIO}}]\label{First theorem by Giorgio dall'aglio} Let $\mu,\nu \in P_{p}(\R)$ with associated distribution functions $F$ and $G$, respectively, and $H$ be a joint distribution function with margins $F$ and $G$. Then for all $p> 1$, the integral
\begin{align*}
I(H)=\int_{-\infty}^{\infty}\int_{-\infty}^{\infty}\lvert y-x\rvert^{p} dH(x,y) 
\end{align*}
is minimized by the function
\begin{equation*}
    M(F(x),G(y))=\min\{F(x),G(y)\}, 
\end{equation*}
i.e.,
 \begin{align*}
     W_p^p(\mu,\nu)=\inf_H I(H)=I(M)=\int_{-\infty}^{\infty}\int_{-\infty}^{\infty}\lvert y-x\rvert^{p} dM(F(x),G(y)) 
 \end{align*}
\end{Theorem}

\begin{proof}
  Obviously, we have $G(y)\geq H(x,y)$ for all $x$ and $y$. Therefore, the expression
 \begin{align*}
  p (p -1)\int_{-\infty}^{\infty}\int_{y}^{\infty}[G(y)-H(x,y)](x-y)^{p -2} dx dy
 \end{align*}
 is minimized when $H$ is maximized, in other words, when $H$ is equal to the Fr\'echet-Hoeffding upper bound $M$.
 
 By the same argument, using $F(x)\geq H(x,y)$, we get that
 \begin{align*}
  p (p -1)\int_{-\infty}^{\infty} \int_{x}^{\infty}[F(x)-H(x,y)](y-x)^{p -2} dydx
 \end{align*}
 is minimized when $H$ is maximized by the Fr\'echet-Hoeffding upper bound $M$.
 
 Putting these arguments together and using Proposition \ref{First proposition by Giorgio dall'aglio}, we get
 \begin{align*}
     W_p^p(\mu,\nu)=\inf_H I(H)=I(M)=\int_{-\infty}^{\infty}\int_{-\infty}^{\infty}\lvert y-x\rvert^{p} dM(F(x),G(y)) 
 \end{align*}
 and it follows as before from Proposition \ref{comonotonicity result} that 
 \begin{align*}
     W_{p}^{p}(\mu,\nu)=\E[|X-Y|^{p}]=\int_{0}^{1}|F^{-1}(u)-G^{-1}(u)|^{p} du
 \end{align*}
 Where $X$ and $Y$ are two comonotone random variables on $(\Omega, \mathcal{F}, \P)$ with joint distribution function $M(F(x),G(y))$.
\end{proof}

%\textcolor{green}{Our proof partially relies on similar ideas as \cite{BDS}. For the case that the shared copula $C$ is not the comonotonicity copula, we finish by applying the representation from Theorem \ref{thm:bds}. In case that $C=M$, we can give an explicit construction relying on classical results from copula theory, see \cite{Nelsen}.}\\
Now we prove Proposition \ref{Representation for measures with the copula M}.
\begin{proof}[Proof of Proposition \ref{Representation for measures with the copula M}]
Let $d>1$ and  $p \in [1,\infty)$. Denote by $\pi_i$ the two-dimensional margins of $\pi$ in $x_i$ and $y_i$, i.e., for any Borel sets $A,B\in\cB(\R)$,
\begin{displaymath}
 \pi_i(A,B)=\pi(\R\times\dotsb\times\R\times A\times\R\times\dotsb\times\R\times B\times\R\dotsb\times\R)
\end{displaymath}
where $A$ and $B$ appear on the $i$-th and $d+i$-th coordinate, respectively. Furthermore, for $\mu,\nu\in P_p(\R^d)$ sharing the same comonotonicity copula $M$, we denote by $\mu_i$ and $\nu_i$ their one-dimensional margins on the $i$-th coordinate. By Proposition \ref{Projection} 
\begin{align*}
    W_{p}^{p}(\mu,\nu)&=\inf_{\pi \in \Pi(\mu,\nu)}\int_{\R^{d}} \int_{\R^{d}} \left \|  x-y\right \|_{p}^{p} \pi(dx, dy)
    \\
    &= \inf_{\pi \in \Pi(\mu,\nu)}  \int_{\R^{d}} \int_{\R^{d}} \sum_{i=1}^{d} |x_{i}-y_{i}|^{p} \pi(dx, dy)
    \\ 
    &\geq  \sum_{i=1}^{d} \inf_{\pi \in \Pi(\mu,\nu)} \int_{\R^{d}} \int_{\R^{d}} |x_{i}-y_{i}|^{p} \pi(dx, dy)
    \\
    &\geq\sum_{i=1}^{d} \inf_{\substack{\pi\in P(\R^{2d})\colon\\\pi_{i} \in \Pi(\mu_i,\nu_i)}} \int_{\R^d} \int_{\R^d} |x_{i}-y_{i}|^{p} \pi(dx, dy)
    \\
    &=\sum_{i=1}^{d} \inf_{\pi_{i} \in \Pi(\mu_i,\nu_i)} \int_{\R} \int_{\R} |x_{i}-y_{i}|^{p} \pi_{i}(dx_{i}, dy_{i})
    \\
    &= \sum_{i=1}^{d} \int_{\R} \int_{\R} |x_{i}-y_{i}|^{p} dM(F_{i}(x_{i}), G_{i}(y_{i})),
\end{align*}
where the last step holds by Theorem \ref{main theorem}. Next, denote by $H_{\mu}$ and $H_{\nu}$ the joint distribution functions induced by $\mu$ and $\nu$, respectively, i.e., for all $(x_{1},\dotsc,x_{d}, y_{1},\dotsc,y_{d}) \in \R^{2d}$,
\begin{align*}
    H_{\mu}(x_{1},\dotsc,x_{d})&=\mu\big((-\infty,x_{1}]\times\dotsb\times (-\infty,x_{d}]\big) = M(F_{1}(x_{1}),\dotsc,F_{d}(x_{d}))
    \\
    H_{\nu}(y_{1},\dotsc,y_{d}) &= \nu\big((-\infty,y_{1}]\times\dotsb\times (-\infty,y_{d}]\big) = M(G_{1}(y_{1}),\dotsc,G_{d}(y_{d})),
\end{align*}
and set
\begin{align*}
    H_{\mu,\nu}(x_{1},\dotsc,x_{d},y_{1},\dotsc,y_{d}):= M(H_{\mu}(x_{1},\dotsc,x_{d}), H_{\nu}(y_{1},\dotsc,y_{d})).
\end{align*}
Then
\begin{align*}
    H_{\mu,\nu}(\infty,\dotsc,\infty,x_{i},\infty,\dotsc,\infty,y_{i},\infty,\dotsc,\infty)=M(F_{i}(x_{i}), G_{i}(y_{i})).
\end{align*}
Furthermore, the copula corresponding to $H_{\mu,\nu}$ by Sklar's Theorem \ref{Sklar's theorem} is given by
\begin{displaymath}
 C_{\mu,\nu}(u_1,\dotsc,u_d,v_1,\dotsc,v_d)=M(M(u_1,\dotsc,u_d),M(v_1,\dotsc,v_d)).
\end{displaymath}
This function $C_{\mu,\nu}\colon[0,1]^{2d}\to[0,1]$ is in fact a copula by \cite[Theorem 3.5.3]{Nelsen}. Therefore, again by Theorem \ref{Sklar's theorem}, $H_{\mu,\nu}\colon\R^{2d}\to[0,1]$ is a $2d$-joint distribution function, which induces a coupling of $\mu$ and $\nu$, which can be seen as follows:
\begin{align*}
    H_{\mu,\nu}(x_1,\dotsc,x_d,\infty,\dotsc,\infty)&=M(H_{\mu}(x_{1},\dotsc,x_{d}),H_\nu(\infty,\dotsc,\infty))
    \\
    &= M(H_{\mu}(x_{1},\dotsc,x_{d}),1)
    \\
    &= H_{\mu}(x_{1},\dotsc,x_{d}).
\end{align*}
Similarly,
\begin{align*}
    H_{\mu,\nu}(\infty,\dotsc,\infty,y_1,\dotsc,y_d)=H_{\nu}(y_{1},\dotsc,y_{d}).
\end{align*}
Going back to the proof, we have 
\begin{align*}
    W_{p}^{p}(\mu,\nu) &\geq \sum_{i=1}^d \int_{\R} \int_{\R} |x_{i}-y_{i}|^{p} dM(F_{i}(x_{i}), G_{i}(y_{i}))
    \\
    &= \sum_{i=1}^d \int_{\R^{d}} \int_{\R^{d}} |x_{i}-y_{i}|^{p} dH_{\mu,\nu}(x,y)
    \\
    &\geq \inf_{\pi \in \Pi(\mu,\nu)} \int_{\R^{d}} \int_{\R^{d}} \left \| x-y \right \|_{p}^{p} \pi(dx dy) = W_{p}^{p}(\mu,\nu)
\end{align*}
%The last equation of the theorem holds directly from Theorem \ref{main theorem} as follows: By above calculation together with Theorem \ref{main theorem}, we get
%\begin{displaymath}
% W_p^p(\mu,\nu)=\sum_{i=1}^d\int_\R\int_\R|x_i-y_i|^pdM(F_i(x_i),G_i(y_i))=\sum_{i=1}^d\int_0^1|F_{i}^{-1}(u)-G_{i}^{-1}(u)|^pdu.
%\end{displaymath}
\end{proof}

\begin{proof}[Proof of Theorem \ref{thm:to_be_deleted}]
%Let $d>1$ and  $p \in [1,\infty)$. Denote by $\pi_i$ the two-dimensional margins of $\pi$ in $x_i$ and $y_i$, i.e., for any Borel sets $A,B\in\cB(\R)$,
%\begin{displaymath}
% \pi_i(A,B)=\pi(\R\times\dotsb\times\R\times A\times\R\times\dotsb\times\R\times B\times\R\dotsb\times\R)
%\end{displaymath}
%where $A$ and $B$ appear on the $i$-th and $d+i$-th coordinate, respectively. Furthermore, for $\mu,\nu\in P_p(\R^d)$, we denote by $\mu_i$ and $\nu_i$ their one-dimensional margins on the $i$-th coordinate. By Proposition \ref{Projection}
%\begin{align*}
 %   W_{p}^{p}(\mu,\nu)&=\inf_{\pi \in \Pi(\mu,\nu)}\int_{\R^{d}} \int_{\R^{d}} \left \|  x-y\right \|_{p}^{p} \pi(dx, dy)
 %   \\
 %   &= \inf_{\pi \in \Pi(\mu,\nu)}  \int_{\R^{d}} \int_{\R^{d}} \sum_{i=1}^{d} |x_{i}-y_{i}|^{p} \pi(dx, dy)
  %  \\ 
 %   &\geq  \sum_{i=1}^{d} \inf_{\pi \in \Pi(\mu,\nu)} \int_{\R^{d}} \int_{\R^{d}} |x_{i}-y_{i}|^{p} \pi(dx, dy)
 %   \\
  %  &\geq\sum_{i=1}^{d} \inf_{\substack{\pi\in P(\R^{2d})\colon\\\pi_{i} \in \Pi(\mu_i,\nu_i)}} \int_{\R} \int_{\R} |x_{i}-y_{i}|^{p} \pi(dx, dy)
   % \\
  %  &=\sum_{i=1}^{d} \inf_{\pi_{i} \in \Pi(\mu_i,\nu_i)} \int_{\R} \int_{\R} |x_{i}-y_{i}|^{p} \pi_{i}(dx_i, dy_i)
 %   \\
  %  &= \sum_{i=1}^{d} \int_{\R} \int_{\R} |x_{i}-y_{i}|^{p} dM(F_{i}(x_{i}), G_{i}(y_{i})) \\ 
  %  &= \sum_{i=1}^{d}  W_{p}^{p}(\mu_{i},\nu_{i})=  W_{p}^{p}(\mu,\nu) ,
%\end{align*}
%where the last equality follows from Theorem \ref{main theorem}.
Following the first calculation of the proof of Proposition \ref{Representation for measures with the copula M} and applying \cite[Theorem 5]{BDS}, we have 
\begin{align*}
   W_{p}^{p}(\mu,\nu)&=\sum_{i=1}^{d} \int_{\R} \int_{\R} |x_{i}-y_{i}|^{p} dM(F_{i}(x_{i}), G_{i}(y_{i}))
   \\ 
    &= \sum_{i=1}^{d}  W_{p}^{p}(\mu_{i},\nu_{i})=  W_{p}^{p}(\mu,\nu).
\end{align*}
By above calculation together with Theorem \ref{main theorem}, we get
\begin{displaymath}
 W_p^p(\mu,\nu)=\sum_{i=1}^d\int_\R\int_\R|x_i-y_i|^pdM(F_i(x_i),G_i(y_i))=\sum_{i=1}^d\int_0^1|F_{i}^{-1}(u)-G_{i}^{-1}(u)|^pdu.
\end{displaymath}
\end{proof}

%\textcolor{red}{What is $\mathcal{X}$ in the following corollary, and how is it a corollary at all? Corollary means it follows easily from another theorem, which one is meant? Since it is part of Theorem \ref{main theorem}, we don't need to formulate it explicitly like this. And it is already included in the proof of \ref{main theorem}, so there is no need to formulate it again.}
%\begin{Corollary}
%For all $\mu,\nu \in P_{p}(\mathcal{X})$, 
%\begin{align*}
%    W_{p}^{p}(\mu,\nu)=\int_{0}^{1}|F^{-1}(u)-G^{-1}(u)|^{p} du.
%\end{align*}
%\end{Corollary}
%\begin{proof}
%  The result follows direct from Proposition \ref{comonotonicity result}.
%\end{proof}

%\begin{Theorem}
 %   Let $\mu,\nu \in P_{p}(\mathcal{X})$, then there exists two random vectors %$X=(X_{1},...,X_{n})$ and $Y=(Y_{1},...,Y_{n})$ such that $\text{Law}(X)=\mu$ and $\text{Law}(Y)=\nu$. let $H$ be the joint distribution function of the random vector $(X,Y)$ with margins $F=(F_{1},...,F_{n})$ and $G=(G_{1},..., G_{n})$ respectively, then by Sklar's theorem there exists a Copula $C :[0,1]^{2d} \rightarrow \R^{d}$ such that for all $x=(x_{1},...,x_{n}) \in \Bar{\R}^{n}$
 %   \begin{align*}
  %      H(x_{1},...,x_{d}, y_{1},...,y_{d})=C(F_{1}(x_{1}),F_{2}(x_{2}),..., F_{d}(x_{d}), G_{1}(y_{1}), G_{2}(y_{2}),...,G_{n}(y_{n}))
   % \end{align*}
%\end{Theorem}

{\bf Acknowledgments.} We thank Dennis Schroers (University Bonn) for very fundamental and necessary remarks to a first version of this article. We thank also Stefano Bonaccorsi (University Trento) for giving to our attention the reference \cite{DALL’AGLIO} at the beginning of this work and for many important comments. 

\appendix
%\section{Proofs of Lemma \ref{Equivalence} and Proposition \ref{comonotonicity result}}\label{app:a}

\section{Proof of Proposition \ref{comonotonicity result}, marginal integration}\label{app:a}

%We start with the proof of Lemma \ref{Equivalence}.
%\begin{proof}[Proof of Lemma \ref{Equivalence}]
%\begin{itemize}
%      \item $(\Leftarrow)$ The first implication follows from the definition of $F^{-1}$.
%      \item $(\Rightarrow)$ Define $A:=\{ t\in \R: F(t) \geq p\} \neq \emptyset$, then there exists a sequence $\{x_{n}\}_{n \in \N}$ such that $x_{n} \in A$ and $x_{n}$ tends to $\inf A=F^{-1}(p), p\in [0,1]$. By definition of right continuity we get that $F(x_{n})$ converges to $F(F^{-1}(p))$. Now since $x_{n} \in A$, we get $F(x_{n})\geq p$, so 
%      \begin{align*}
%         F(F^{-1}(p)) =\lim_{n\to\infty} F(x_{n}) \geq p.
%      \end{align*}
%  Considering our assumption, $F^{-1}(p) \leq x$, since $F$ is a non-decreasing function we obtain 
%  \begin{align*}
%      F(x)\geq F(F^{-1}(p)) \geq p.
%  \end{align*}
%  \end{itemize} 
%\end{proof}
Here we prove Proposition \ref{comonotonicity result}.
\begin{proof}[Proof of Proposition \ref{comonotonicity result}]
   By using $(b)$ of Theorem \ref{Equivalent conditions comonotonicity}, we can write, for $U \sim \cU(0,1)$
   \begin{align*}
        \E[g(X,Y)]&= \E \left [ g(F^{-1}(U), G^{-1}(U))  \right ].
   \end{align*}
   By letting $\phi$ be the distribution function of $U$, we get
   \begin{align*}
       \E \left [ g(F^{-1}(U), G^{-1}(U))  \right ]&=\int_{-\infty}^{\infty} g(F^{-1}(u),G^{-1}(u)) d\phi(u) \\ &= \int_{0}^{1} g(F^{-1}(u), G^{-1}(u)) du.
   \end{align*}
    The last equation follows since $U$ is uniformly distributed on $(0,1)$.
\end{proof}

%Finally, we explain the connection between distribution functions and probability measures on $\R^d$.
Finally, we comment on how the integration of a function $f\colon\R^d\to\R$ with respect to $\R^{2d}$ can be reduced to the integration w.r.t. the corresponding marginal on $\R^d$.
\begin{Proposition}\label{Projection}
 %Let $H$ be a joint distribution function on $\R^{d} \times \R^{d}$. Then $H$ induces a probability measure $\pi$ on $\R^{d} \times \R^{d}$. 
 Let $\mu$ and $\nu$ be the margins of $\pi$ such that $\mu(A)=\pi(A \times \R^{d})$ and $\nu(B)=\pi(\R^{d} \times B)$ for all $A,B \in \mathcal{B}(\R^{d})$. Then for any integrable function $f\colon\R^d\to\R$,
 \begin{equation}\label{eq:appendix_margins}
   \int_{\R^{d} \times \R^{d}} f(x) \pi(dx,dy)=\int_{\R^{d}} f(x) \mu(dx)
 \end{equation}
\end{Proposition}
\begin{proof}
The statement is proven in multiple steps. We start by proving \eqref{eq:appendix_margins} for indicator functions, then continue with increasingly complex functions.
\begin{enumerate}
 \item Let $f=1_{A}, A \in \mathcal{B}(\R^{d})$
 \begin{align*}
  \int_{\R^{d} \times \R^{d}} f(x) \pi(dx,dy)&= \int_{\R^{d} \times \R^{d}} 1_{A}(x) \pi(dx,dy)
  \\
  &= \int_{A \times \R^{d}} \pi(dx,dy)
  \\
  &= \pi(A \times \R^d) = \mu(A)
  \\
  &= \int_{A} \mu(dx)=\int_{\R^{d}} 1_{A}(x) \mu(dx) \\
  &= \int_{\R^{d}} f(x) \mu(dx).
 \end{align*}
 \item Let $f=\sum_{i=1}^{n} \alpha_{i} 1_{A_{i}}, \alpha_i \in \R, A_{i} \in \mathcal{B}(\R^{d})$. Then
 \begin{align*}
  \int_{\R^{d} \times \R^{d}} f(x) \pi(dx,dy)&= \int_{\R^{d} \times \R^{d}} \sum_{i=1}^{n} \alpha_{i} 1_{A_{i}}(x) \pi(dx,dy)
  \\
  &= \sum_{i=1}^{n} \alpha_{i} \int_{A_{i} \times \R^{d}} \pi(dx,dy)
  \\
  &= \sum_{i=1}^{n} \alpha_{i}\pi[A_{i} \times \R^{d}]= \sum_{i=1}^{n} \alpha_{i}\mu(A_{i})
  \\
  &= \sum_{i=1}^{n} \alpha_{i}\int_{A_{i}} \mu(dx)= \int_{\R^{d}} \sum_{i=1}^{n} \alpha_{i} 1_{A_{i}}(x) \mu(dx)
  \\
  &= \int_{\R^{d}} f(x) \mu(dx).
 \end{align*}
 \item Let $f \geq 0$ be a nonnegative integrable function. Then there exist a sequence of elementary functions $f_{k}=\sum_{i=1}^{n_{k}} \alpha_{i}^{k} 1_{A_{i}^{k}}$ such that $f_{k} \uparrow f$. By monotone convergence and step (2), we get
 \begin{align*}
     \int_{\R^d\times\R^d}f(x)\pi(dx,dy)&=\lim_{k\to\infty}\int_{\R^d\times\R^d}f_k(x)\pi(dx,dy)
     \\
     &=\lim_{k\to\infty}\int_{\R^d}f(x)\mu(dx)
     \\
     &=\int_{\R^d}f(x)\mu(dx)
 \end{align*}
 \item Let $f\in L^1(\R^d)$ arbitrary. Then there exists a decomposition $f=f_+-f_-$ into nonnegative integrable functions $f_+,f_+\geq 0$. By applying Step (3) to the decomposition, equation \eqref{eq:appendix_margins} can be shown for these functions as well.
 \end{enumerate}
\end{proof}

\bibliographystyle{amsplain}
\phantomsection\addcontentsline{toc}{section}{\refname}\bibliography{Bibliography}

\end{document}